\documentclass{amsart}

\usepackage{latexsym,amsmath,amsfonts,amscd,amssymb, amsthm}

\theoremstyle{plain}
\newtheorem{theorem}{Theorem}[section]

\newtheorem{proposition}{Proposition}[section]
\newtheorem{corollary}{Corollary}[section]
 
\theoremstyle{definition}

\theoremstyle{remark}
\newtheorem{remark}{Remark}[section]

\title{Pfister's theorem fails in the Hermitian case}

\author{John P. D'Angelo}

\address{Dept. of Mathematics, Univ. of Illinois, 1409 W. Green St., Urbana IL 61801}

\email{jpda@math.uiuc.edu}

\author{Ji\v{r}\'\i\ Lebl}

\address{Dept. of Mathematics, Univ. of Illinois, 1409 W. Green St., Urbana IL 61801}

\email{jlebl@math.uiuc.edu}

\begin{document}

\maketitle

\begin{abstract} We show that the Hermitian analogue of a famous result
of Pfister fails. To do so we provide a Hermitian symmetric polynomial $r$ of total degree $2d$ such that any
non-zero multiple of it cannot be written as a Hermitian sum of squares
with fewer than $d+1$ squares.

\medskip

\noindent
{\bf AMS Classification Numbers}: 12D15, 14P05, 15B57, 32V15.

\medskip

\end{abstract}

\section{Introduction}

Artin's solution of Hilbert's $17$-th problem [A] includes the following statement. Let $r$ be a polynomial in $n$ real variables.
Then $r \ge 0$ if and only if there is a polynomial $q$, not identically $0$, such that $q^2 r$ is a sum of squares of polynomials. See also [S] for recent developments.
Pfister [Pf] proved that we may always choose $q$ such that the number of terms in the sum
is at most $2^n$. This result is remarkable because the number of terms (the length of the sum
of squares of $q^2 r$) is independent of the degree of $r$.

See [Q], [D2], [D3] and their references for Hermitian analogues of Hilbert's problem.
Related work by both authors ([D1], [DL], [L])
on applications of Hermitian symmetric polynomials to CR geometry have led us to a simple counterexample
to the natural Hermitian analogue of Pfister's result. 

Let $r(z, {\overline w})$ be a polynomial on ${\bf C}^n \times{\bf C}^n$.  Using multi-index notation we write
$$ r(z,{\overline w}) = \sum c_{\alpha, \beta} z^\alpha {\overline w}^\beta. $$
We let ${\bf r}(r)$ denote the rank of the matrix $c_{\alpha, \beta}$.
The function $z \to r(z,{\overline z})$ is real-valued if and only if this matrix  is Hermitian.
The function $z \to r(z,{\overline z})$ is a {\it squared norm} or {\it Hermitian sum of squares}
if and only if this matrix is non-negative definite. In this case there are polynomials $p_1(z),...,p_k(z)$ for which
$$ r(z,{\overline z}) = \sum_{j=1}^k |p_j(z)|^2 = ||p(z)||^2. \eqno (1) $$
By linear algebra it follows that ${\bf r}(r)$ is the minimum $k$ for which (1) holds.
Hence one might call the rank of a squared norm its {\it Hermitian length}.

Unlike in the real case, not every non-negative Hermitian polynomial $r$ divides a squared norm.
Suppose however that $r$ does so; in other words, assume that there is a polynomial $s$, not
identically $0$, such that $rs = ||p||^2$. 
We naturally ask what bounds are possible on the rank of $||p||^2$.
We prove below that there is no bound independent of the degree of $r$. In particular, the Hermitian analogue of
Pfister's result fails.

We give the simple example now in Corollary 1.1; we also state Theorem 1.1 from which the result follows. We prove Theorem 1.1
in the next section. Corollary 1.1 shows that the Hermitian length of every nonzero squared norm
divisible by $r$ is at least $d+1$, and hence the Hermitian analogue of Pfister's result fails.

\begin{corollary} Put $n=1$ and set $r(z,{\overline z})= (1+ |z|^2)^d$. Assume that $||p||^2$
is a multiple of $r$ and that $p$ is not identically $0$.
Then ${\bf r}(||p||^2) \ge d+1$, and equality is possible. \end{corollary}

Equality holds with $r$ itself. Expanding by the binomial theorem writes $r$
as a squared norm with rank $d+1$. The following stronger result holds in arbitrary dimensions and immediately implies the Corollary.
We write $||z||^2 = \sum_{j=1}^n |z_j|^2$. 

\begin{theorem} Set $r(z,{\overline z})= (1+ ||z||^2)^d$ and let $g$ be a nonzero
multiple of $r$. Then ${\bf r}(g) \ge {n+d \choose d}$, and equality is possible. 
Here ${n+d \choose d}=M(n,d)$ equals the dimension of the vector space of polynomials 
of degree at most $d$ in $n$ variables. \end{theorem}

Alternatively we can bihomogenize $r$ and use $||Z||^{2d}$, where $Z=(z_1,...,z_n,z_{n+1})$.
The restated conclusion is then that a nonzero multiple of $||Z||^{2d}$ must have rank at least $N(n+1,d)$,
where $N(n+1,d)$ denotes the dimension of the space of homogeneous polynomials of degree $d$ in $n+1$ variables.
Note that $N(n,k)$ equals the binomial coefficient ${n+k-1 \choose k}$, 
namely the rank of the function $||z||^{2k}$.
Note also that 
$$M(n,d) = \sum_{k=0}^d N(n,k) = N(n+1,d). $$ 

The homogenized version of Theorem 1.1 holds in the real-analytic case as well. See Theorem 2.1, which
generalizes a well-known lemma of Huang. Huang [H] proved the following.
Let $f_1,...,f_k, g_1,...g_k$ be holomorphic functions defined near the origin in ${\bf C}^m$ and vanishing there,
and suppose that the expression 
$$\sum_{j=1}^k f_j(z) {\overline {g_j(z)}} \eqno (2) $$ 
is divisible by $||z||^2$. If (2) is not
identically zero, then $k \ge m$. Huang's lemma is equivalent to the special case
of Theorem 2.1 when $d=1$.

The authors acknowledge support from NSF grants DMS 07-53978 (JPD) and DMS 09-00885 (JL). They also
wish to thank AIM for the workshop on CR Complexity Theory in 2010; preparing for that workshop helped lead
us to this result. Finally we thank Martin Harrison and the referee for 
pointing out several places where the exposition needed improvement.

\section{Proof of Theorem 1.1}

In this paper we assume the polynomials used have complex coefficients. We note
however that Proposition 2.1 below
holds for polynomials over an arbitrary field of characteristic zero.
First we prove the analogue of Theorem 1.1 when the matrix of coefficients is diagonal.
To finish the proof of Theorem 1.1 we reduce the general case
to the diagonal case.

\begin{proposition} Put $s(x)= \sum_{j=1}^n x_j$.
Let $p(x)$ be a homogeneous polynomial and suppose $p$
is a multiple of $s^d$. Then either $p=0$ or $p$ has at least $N(n,d)$ monomials.
\end{proposition}

\begin{proof} We first observe that the result is trivial when $n=1$, as $N(1,d)=1$ for all $d$.
We next consider the case $n=2$. Note that $N(2,d)=d+1$ for all $d$.
After dehomogenizing, it suffices to prove the following statement in one variable, for which
we have found two proofs. If the polynomial $p$ defined by $p(x) = (1+x)^d q(x)$ is not identically zero, 
then $p$ has at least $d+1$ terms. 

The first proof is by the method of descent. Suppose that there is an integer $d$ and polynomials $q$
and $r$ such that $r(x) = (1+x)^d q(x)$, and such that $q$ has at most $d$ terms.
Then there is a smallest such $d$. If the resulting polynomial $q$ is divsible by $x$, then $r$ also is, and we
divide both sides by $x$. We may therefore assume that either $q$ is identically zero, or that $q$ has a nonzero constant term.
In the second situation, differentiate both sides to obtain

$$ r'(x) = (1+x)^{d-1} ( d q(x) + (1+x) q'(x)). $$
Now $r'$ has at most $d-1$ terms, and it is divisible by $(1+x)^{d-1}$.
Hence there is an example with $d$ replaced by $d-1$. Since there is no example for $d=0$ other than $r$ being identically $0$,  we conclude that
there is no $d$ at all for which $r$ is not identically $0$.

The second proof is more complicated. Let $p$ be an arbitrary polynomial of degree $m+d$. We write $p(x)$ in two ways:

$$ p(x) = \sum_{k=0}^{m+d}  a_k x^k = \sum_{k=0}^{m+d} c_k (1+x)^k.  \eqno (3) $$
The two formulas amount to different choices of basis in the space of polynomials of degree $m+d$.
The mapping $L:{\bf C}^{m+d+1} \to {\bf C}^{m+d+1}$
taking the column vector $c=(c_0,..., c_{m+d})$ into $(a_0,..., a_{m+d})$ is linear.
The entry in the $j$-th row and $k$-th column of its matrix is
the binomial coefficients ${k \choose j}$, for $0 \le j,k \le m+d$, where as usual we set ${j \choose k}$ equal to $0$ if $j<k$. We also set ${0 \choose k} =1$. 
The matrix of $L$ is upper triangular, and all diagonal entries are equal to $1$.
Thus $L$ is invertible. The condition that $p$ be divisible by $(1+x)^d$ amounts to saying that $c_j=0$ for $0 \le j < d$. We therefore consider the matrix $L'$ obtained from $L$ by deleting the first $d$ columns. We claim that any square submatrix of $L'$ of size $m+1$ is invertible.
We omit the details of this claim; in fact the best proof of the claim is to use the first proof above.

If $p$ had fewer that $d+1$ terms, then at least $m+d+1 - d = m+1$ of the $a_j$ would vanish. 
Hence there is an $m+1$ by $m+1$ submatrix of $L'$
annihilating the column vector $c' = (0, c_d,...,c_{m+d})$. Since such matrices are invertible,
we obtain $c'= 0$. Hence $c=0$, and therefore $p=0$.
Therefore the only element divisible by $(1+x)^d$ with fewer than $d+1$ terms is $0$.

We use this result as an induction step. Assume that we have proved the result in dimension $n$. Let
$ y= (y_1,...,y_n)$, and put $s=\sum_{k=1}^n y_k$. The induction hypothesis guarantees that a nonzero polynomial divisible
by $s^j$ has at least $N(n,j)$ terms.

Let $p$ be homogeneous of degree $m+d$ in the $n+1$ variables $(y,x)$. Assume
that $p$ is a multiple of $(x+s)^d$. We wish to show that the number of distinct monomials in $p$ is at least
$N(n+1,d)$. By dehomogenization we have

$$ N(n+1,d) = \sum_{j=0}^d  N(n,j). \eqno (4) $$
We expand $p$ in two ways, writing

$$ p(x,y) = (x+s)^d \left(\sum_{j=0}^m h_j(y)x^j\right) = \sum_{j=0}^{m+d} 
A_j(y)x^j
\eqno (5) $$

First we note that, after dividing through by a power of $x$, we may 
assume without loss of generality that $h_0(y) \ne 0$,
and hence that $ A_0(y) \ne 0$.

Next we claim that at least $d+1$ of the $A_j$ in (5) are not zero.
To verify the claim, replace $x$ by $ws$ in (5).
Using homogeneity, we obtain
a polynomial in the single variable $w$  that is divisible by
$ (1+w)^d$. Hence the claim follows from the one variable case proved 
above. Hence there exist $d+1$ integers such that
$$  0 = j_0 < j_1 ... < j_d \eqno (6) $$
and for which $A_{j_k} \ne 0$. At each stage
we choose the integers in (6) minimally.

Next we note that $s^{d-k}$ divides $A_{j_k}$.
This result holds by expanding the middle term in (5)
by the binomial theorem, which yields an explicit formula
for the $A_j(y)$, and then proceeding inductively using
the minimality.

Each expression $A_{j_k}(y)x^{j_k}$ 
is divisible by a different power
of $x$ and hence all the resulting  terms are distinct.  Therefore
if $K$ is the number of terms in $p(x,y)$ then
$$ K \geq
\sum_{k=0}^d N(n,k) = N(n+1,d) . \eqno (7) $$
The inequality $K \ge N(n+1,d)$ from (7) completes the induction step. \end{proof}

In the proof we proved the following statement.
 If $P(t)$ is a nonzero multiple of $(1+t)^d$, 
then $P$ has at least $d+1$ nonzero terms. We then used this 
result in the inductive step.

Proposition 2.1 is the special case of the general situation
when the matrix of coefficients is diagonal. It is somewhat analogous to the degree estimates proved in [DLP].
To pass from the diagonal case to the general case we replace the number of monomials occurring in a polynomial with the rank
of the polynomial. We recall the needed linear algebra.

Let $R$ be a real-analytic function defined near the origin in ${\bf C}^n$.
Near the origin we write

$$ R(z, {\overline z}) = \sum _{a,b} c_{\alpha,\beta} z^a {\overline z}^b. \eqno (8) $$
Its rank ${\bf r}(R)$ 
is defined to be the rank of the possibly infinite matrix of coefficients $(c_{\alpha,\beta})$. 
The rank is the minimum number $k$  of linearly independent local holomorphic
functions $f_j$ and $g_j$ for which we can write

$$ R(z,{\overline z}) = \sum_{j=1}^k f_j(z) {\overline {g_j(z)}}. \eqno (9) $$
We allow $k$ to take the value $\infty$. 

We clarify the connection with the diagonal case. Suppose, for each $j$ that there is a multi-index $\alpha_j$
such that
$f_j(z) = c_j z^{\alpha_j}$ and $g_j(z) = z^{\alpha_j}$. After setting $x_j= |z_j|^2$ and using multi-index notation,
we can rewrite (9) in the form
$$ R(x) = \sum_{j=1}^{k} c_j x^{\alpha_j} \eqno (10) $$
where the $c_j$ are non-zero constants and the $\alpha_j$ are distinct multi-indices. 
Then ${\bf r}(R)$ equals the number of nonzero monomials in (10).

The following result includes Theorem 1.1 as a special case when $R$ is a polynomial.

\begin{theorem} Let $R(z, {\overline z})$ be a real-analytic function defined near the origin in ${\bf C}^n$.
Suppose that $R$ is not identically zero, and that $R$ is a multiple of $||z||^{2d}$. Then the rank of $R$ is at least
$N(n,d)$. Equality is possible: for example equality holds when $R=||z||^{2d}$. \end{theorem}

\begin{proof} When $n=1$ the result is trivial, as $N(1,d)=1$ for all $d$, and we are saying only that $R$ is not identically $0$.
The special case $d=1$ corresponds to Huang's lemma,  but our argument is considerably different even in this case.

Write $R(z,{\overline z}) = ||z||^{2d} r(z, {\overline z})$, where $r$ is real-analytic in some neighborhood of the origin.
Consider the lowest order part $u$ of the Taylor expansion of $r$ at the origin. Note that the lowest order part
of $||z||^{2d}r$ is given by $||z||^{2d} u$. Since the matrix of coefficients of $||z||^{2d} u$ is a submatrix of the
matrix of coefficients of $R$, we get ${\bf r}(||z||^{2d}r) \ge {\bf r}(||z||^{2d}u)$.

We may therefore assume that
$$ R(z, {\overline z}) = ||z||^{2d} u(z, {\overline z}), \eqno (11) $$
where $u$ is homogeneous in the $z$ and ${\overline z}$ variables.
We write
$$ u(z, {\overline z}) = \sum_{|\mu|+|\nu|=m} c_{\mu,\nu} z^\mu {\overline z}^\nu. \eqno (12) $$ 

Now that we are in the polynomial case it is convenient to dehomogenize. We use different notation.
Assume that $z \in {\bf C}^{n-1}$ and put
$$ p(z,\bar{z}) = r(z,\bar{z}) (1 + ||z||^2)^d. $$
We must show that $p$ cannot be written as a squared norm with fewer terms than $M(n-1,d)= N(n,d)$. In other words, we wish to find
a lower bound on the rank of the matrix of coefficients of $p$.

We decompose $p(z, {\overline z})$ according to the following formula:
\begin{equation*}
\begin{split}
p(z, {\overline z}) & = 
\left(
\sum_\beta c_{\beta,\beta} z^\beta \bar{z}^\beta
\right)
+
\sum_{\alpha \not= 0}
\left(
\sum_\beta c_{\beta+\alpha,\beta} z^{\beta+\alpha} \bar{z}^\beta
\right)
+
\sum_{\alpha \not= 0}
\left(
\sum_\beta c_{\beta,\beta+\alpha} z^\beta \bar{z}^{\beta+\alpha}
\right)
\\
& =
\left(
\sum_\beta c_{\beta,\beta} \lvert z^\beta \rvert^2
\right)
+
\sum_{\alpha \not= 0}
z^\alpha
\left(
\sum_\beta c_{\beta+\alpha,\beta} \lvert z^\beta \rvert^2
\right)
+
\sum_{\alpha \not= 0}
\bar{z}^\alpha
\left(
\sum_\beta c_{\beta,\beta+\alpha} \lvert z^\beta \rvert^2
\right)
.
\end{split}
\end{equation*}
If only the first sum is nonzero (that is, if the matrix of
coefficients is diagonal), then we are already in the case of Proposition 2.1,
and the conclusion holds. Therefore we assume that one of the other two terms is nonzero,
and without loss of generality we suppose the second term is non-zero. (In the Hermitian case
the terms are conjugates of each other.) The above decomposition is invariant under multiplication
by $1+||z||^2$ and hence by $(1+||z||^2)^d$.  That is, if we
decompose $r$ as above
and multiply each term by $(1+||z||^2)^d$, we get the corresponding decomposition 
of $p$.

Next, impose a monomial order on the multi-indices $\alpha$.  That is,
we have a total well ordering on all monomials that respects multiplication:
if $\alpha < \beta$ then $\alpha + \gamma < \beta + \gamma$.  For example,
lexicographical ordering will suffice.
In this ordering, find the largest $\alpha$ such that
$$ \sum_\beta c_{\beta+\alpha,\beta} \lvert z^\beta \rvert^2 $$
is nonzero.
Let $\beta$ range over those
$c_{\beta+\alpha,\beta}$ that are nonzero.
We note that the vectors
$[ c_{\beta+\alpha,*} ]$ are linearly independent, because
$c_{\beta+\alpha,\gamma}$ must be zero for $\gamma < \beta$
by the extremality of $\alpha$.

Therefore ${\bf r}(p)$, the rank of the matrix of coefficients,
is bounded below by the number of nonzero terms in
$$
\sum_\beta c_{\beta+\alpha,\beta} \lvert z^\beta \rvert^2. \eqno (13)
$$
Since (13) is divisible by $(1+||z||^2)^d$, the inequality ${\bf r}(p) \ge M(n-1,d) = N(n,d)$
follows by Proposition 2.1.
\end{proof}

\begin{remark} There is a subtle difference between the polynomial case
and the power series case. Consider the polynomial $(1 + ||z||^2)^d$ as the germ at $0$ of a real-analytic function.
It is invertible, and hence there is a multiple of it with rank $1$.
This fact does not contradict Theorem 2.1; the conclusion there
in the real-analytic case requires that we work instead with $||z||^2$. In the polynomial case we pass back and forth between
$1+||z||^2$ and $|z_{n+1}|^2 + ||z||^2$ via bihomogenization. \end{remark}

\section{bibliography}

\medskip

[A] Artin, E., \"Uber die Zerlegung definiter Funktionen in Quadrate, Abh. Math. Sem. Univ. Hamburg 5
(1927), 110-115.

\medskip

[D1] D'Angelo, J., Inequalities from Complex Analysis, Carus Mathematical Monograph No. 28, 
Mathematics Association of America, 2002.

\medskip

[D2] D'Angelo, J., Hermitian analogues of Hilbert's $17$-th problem, preprint.

\medskip

[D3] D'Angelo, J., Complex variables analogues of Hilbert's Seventeenth Problem,
{\it International Journal of Mathematics}, Vol. 16, No. 6 (2005), 609-627.

\medskip
[DL] D'Angelo, J. and Lebl, J., Hermitian symmetric polynomials and CR
complexity, {\it Journal Geometric Analysis} (2010), (to appear).

\medskip 

[DLP] D'Angelo, J., Lebl, J.,  and Peters, H.,
Degree estimates for polynomials constant on hyperplanes,  {\it Michigan
Math.\ J.}, 55 (2007), no. 3, 693-713.

\medskip
[H] Huang, X., On a linearity problem for proper maps between balls in complex spaces
of different dimensions, {\it J. Diff.\ Geometry} 51 (1999), no 1, 13-33.

\medskip
[L] Lebl, J., Normal forms, Hermitian operators, and CR maps of spheres and
hyperquadrics,  {\it Michigan Math.\ J.} (to appear), arXiv:0906.0325.

\medskip

[PF] Pfister, A., Zur Darstellung definiter Funktionen als Summe von
Quadraten. {\it Invent.\ Math.} 4 (1967), 229--237.

\medskip

[Q] Quillen, Daniel G., On the Representation of Hermitian Forms
as Sums of Squares, {\it Invent.\ Math.} 5 (1968), 237-242.

\medskip

[S] Scheiderer, Claus, Positivity and sums of squares: a guide to recent results,
pages 271-324 in Emerging applications of algebraic geometry, IMA Vol. Math. Appl., 149, Springer, New York, 2009.

\end{document}